\documentclass[draft]{amsart}
\usepackage{amssymb,graphicx}
\usepackage{multirow,comment,caption}

\newcommand{\df}{\dfrac}
\newcommand{\tf}{\tfrac}

\renewcommand{\i}{\infty}

\newcommand{\beqs}{\begin{equation*}}
\newcommand{\eeqs}{\end{equation*}}
\numberwithin{equation}{section}
 \theoremstyle{plain}
\newtheorem{theorem}{Theorem}[section]
\newtheorem{lemma}[theorem]{Lemma}
\newtheorem{corollary}[theorem]{Corollary}

\theoremstyle{remark}

\newcommand{\noi}{\noindent}

\begin{document}

\makeatletter
\def\imod#1{\allowbreak\mkern10mu({\operator@font mod}\,\,#1)}
\makeatother

\author{Alexander Berkovich}
\address{Department of Mathematics, University of Florida, 496 Little Hall, Gainesville, FL 32611, USA}
\email{alexb@ufl.edu}

\title[\scalebox{.9}{ On the Gauss \emph{E{$\Upsilon$}{PHKA}} theorem and  some allied inequalities}]{ On the Gauss \emph{E{$\Upsilon$}{PHKA}} theorem and  some allied inequalities}

\thanks{}
     
\begin{abstract} 

We use the $1907$ Hurwitz formula along with the Jacobi triple product identity to understand representation properties of two JP (Jones-Pall) forms of Kaplansky:  $9x^2+ 16y^2 +36z^2 + 16yz+ 4xz + 8xy$ and $9x^2+ 17y^2 +32z^2 -8yz+ 8xz + 6xy$. We also  discuss three nontrivial analogues of the Gauss  \emph{E{$\Upsilon$}{PHKA} }theorem.  The technique used can be applied to all known spinor regular ternary  quadratic forms.

\end{abstract}

\dedicatory{}   
   
  
		
\date{\today}
   
\maketitle
   
\section{Introduction}
\label{intro}

We use notation $(a,b,c,d,e,f)$ to represent a positive ternary quadratic form $ax^2+by^2+cz^2+dyz+exz+fxy$. 
We remark that this paper only considers positive ternary quadratic forms. 
We use $(a,b,c,d,e,f;n)$ to denote the total number of representations of $n$ by $(a,b,c,d,e,f)$, 
and take $(a,b,c,d,e,f;n)=0$ when $n \not\in \mathbb{N}$. 
The associated theta series to the form $(a,b,c,d,e,f)$ is
\begin{equation}
\label{thet}
\vartheta(a,b,c,d,e,f,q):=\sum_{x,y,z}q^{ax^2 +by^2 +cz^2 +dyz +exz +fxy} = \sum_{n\geq 0}(a,b,c,d,e,f;n)q^n.
\end{equation}
Here and throughout the paper we assume that $q$ is a complex number with $|q|<1$.
The discriminant $\Delta$ of $(a,b,c,d,e,f)$ is defined as
\[
\Delta:= \frac{1}{2} \mbox{det}\left( \begin{array}{lll}
 2a& f &e  \\
f & 2b &d  \\
 e&d  &2c  \\\end{array} \right) = 4abc +def -ad^2 -be^2 -cf^2.
\]
Note that the discriminant $\Delta>0$ for a positive ternary quadratic form.
Two ternary quadratic forms of discriminant $\Delta$ are in the same genus if they are equivalent 
over $\mathbb{Q}$ via a transformation matrix in $SL(3,\mathbb{Q})$ whose entries have denominators coprime to $2\Delta$.

The following result is used more often than any other in research in the theory of numbers.
\begin{theorem}[Three-square theorem]
\label{LG0}
Let $n,m$ and $v$ be non-negative integers. Then
\begin{equation}
\label{lg0}
(1,1,1,0,0,0;n)\geq0,
\end{equation}  
with equality if and only if $n= 4^v(8m+7)$.
\end{theorem}
\noi
The proof was given by Legendre in \cite{legen} and by Gauss in \cite{gauss}.
Using this result, Gauss proved his celebrated \emph{E{$\Upsilon$}{PHKA}} theorem. According to this  theorem, every natural number is a sum of three  triangular
numbers. This is equivalent to the statement that every natural number of the form  $8m+3$ is a sum of three odd squares.
\begin{theorem}[Gauss]
\label{g0}
Let $n$ be a non-negative integer. Then
\begin{equation}
(1,1,1,0,0,0; 8n+3)=[q^{8n+3}]\sum_{\substack{x\equiv1\imod{2},\\
                y\equiv1\imod{2},\\
							  z\equiv1\imod{2}}} q^{x^2+y^2+z^2} >0.
\end{equation}
\end{theorem}
\noi
Here and throughout the paper $[q^n]F(q)$ denotes the coefficient of $q^n$ extracted from $F(q)$ (when written as a Maclaurin series).

Using the Three-square theorem  along with two easily verifiable identities
\begin{equation}
\label{kap1a}
(1,1,2,0,0,0;2n)=(1,1,1,0,0,0;n), n\in \mathbb{N}, 
\end{equation}
\begin{equation}
\label{kap1b}
3(1,1,2,0,0,0;2n+1)= (1,1,1,0,0,0;4n+2), n\in \mathbb{N}, 
\end{equation}
we can prove the following: 
\begin{theorem}
\label{Dick86}
Let $n,m$ and $v$ be non-negative integers. Then
\begin{equation}
\label{dick86}
(1,1,2,0,0,0;n)\geq0,
\end{equation}  
with equality if and only if $n=4^v2(8m+7)$. 
\end{theorem}

\noi
This result is, essentially,  Dickson's Theorem $86$ in \cite{dickson}. Theorem $89$ in \cite{dickson} can be stated as  
\begin{theorem}
\label{Dick89}
Let $n,m$ and $v$ be non-negative integers. Then
\begin{equation}
\label{dick89}
(1,1,3,0,0,0;n)\geq0,
\end{equation}  
with equality if and only if $n=9^v3(3m+2)$. 
\end{theorem}
\noi
In his $1907$ letter, Hurwitz \cite{hur} proposed that $(1,1,1,0,0,0;n^2)$ may be expressed as a simple finite function of the divisors of $n$. 
In fact, if $n$ is a positive integer with prime factorization
\begin{equation}
n= 2^a\prod_{p>2}p^v.
\end{equation}
Then
\begin{equation}
\label{h1}
\frac{(1,1,1,0,0,0;n^2)}{6}= \prod_{p>2}\left( \frac{1-p^{v+1}}{1-p} -\left(\df{-1}{p}\right)\frac{ 1-p^v}{1-p}\right). 
\end{equation}

\noi
Here and throughout the paper $ \left(\df{a}{n}\right)$ denotes the Jacobi symbol.
A proof of this formula, using integral quaternions, has been given by Pall \cite{pall}.
A completely elementary arithmetical proof, based on the methods of Liouville, was given by Olds \cite{olds}.
Recently, explicit analogues of Hurwitz's formula for many diagonal ternary quadratic forms of class number one 
were established in \cite{cooperlam} and \cite{guo}. 
\begin{theorem}
Let $n$ be a positive integer with prime factorization
\begin{equation}
n= 2^a\prod_{p>2}p^v.
\end{equation}
Then
\begin{equation}
\label{h2}
\frac{(1,1,2,0,0,0;n^2)}{4}= f(a)\prod_{p>2}\left( \frac{1-p^{v+1}}{1-p} -\left(\df{-2}{p}\right)\frac{ 1-p^v}{1-p}\right),
\end{equation}
where $f(0) =1$ and $f(a)=3$ when $a\geq1$. 
\end{theorem}

\begin{theorem}
Let $n$ be a positive integer with prime factorization
\begin{equation}
n= 2^a 3^b\prod_{p>3}p^v.
\end{equation}
Then
\begin{equation}
\label{h3}
\frac{(1,1,3,0,0,0;n^2)}{4}= (2^{a+1}-1)\prod_{p>3}\left( \frac{1-p^{v+1}}{1-p} -\left(\df{-3}{p}\right)\frac{ 1-p^v}{1-p}\right).
\end{equation}
\end{theorem}
\noi
We remark that all analogues of Hurwitz's formula for ternary quadratic forms of class number one 
easily follow from the Siegel formula \cite{siegel}, \cite[(2.1)]{berk}. Moreover, it is not necessary for the class number to be one.
For example, it is not difficult to establish the following:
\begin{theorem}
Let $n$ be a positive integer with prime factorization
\begin{equation}
n= 2^a 3^b\prod_{p>3}p^v.
\end{equation}
Then
\begin{equation}
\label{insert}
\df{(1,3,36,0,0,0;n^2)+(3,4,9,0,0,0;n^2)}{2} =   (3~2^{a}  -2) g(b) \prod_{p>3}\left( \frac{1-p^{v+1}}{1-p} -\left(\df{-3}{p}\right)\frac{ 1-p^v}{1-p}\right),
\end{equation}
where   $g(0) =1, g(b)=2$ when $b\geq1$.
\end{theorem}

Next, we observe  that \eqref{h1}, \eqref{h2}, \eqref{h3} and \eqref{insert} imply the following inequalities:
\begin{corollary}
Let $M\not\equiv 0\imod 2, E\not\equiv 0\imod 2$ and $ W\not\equiv 0\imod 3$  be positive integers. Then
\begin{equation}
\label{hc1}
\df{(1,1,1,0,0,0;M^2)}{6}\ge M,
\end{equation}
with equality if and only if all prime divisors of $M$ are congruent to $1\imod 4$,
\begin{equation}
\label{hc2}
\df{(1,1,2,0,0,0;E^2)}{4} \ge E,
\end{equation}
with equality if and only if all prime divisors of $E$ are congruent to $1\imod 8$ or $3\imod 8$, 
\begin{equation}
\label{hc3}
\df{(1,1,3,0,0,0;W^2)}{4} \ge W,
\end{equation}
\begin{equation}
\label{hc4}
\df{(1,3,36,0,0,0;W^2)+(3,4,9,0,0,0;W^2)}{2} \ge W, 
\end{equation}
with equality if and only if all prime divisors of $W$ are congruent to $1\imod 3$. 
\end{corollary}

One of  main objectives of this paper is to discuss three nontrivial analogues of the Gauss theorem \ref{g0}.
We will use \eqref{hc1} to prove: 
\begin{theorem}
\label{b10}
Let $n$ be a positive integer. Then
\begin{equation}
\label{huj}
[q^{8n+1}]\sum_{\substack{x\equiv1\imod{4},\\
                \mbox{ }y\equiv2\imod{8},\\
                z\equiv2\imod{8}}} q^{x^2+ y^2 + z^2}\geq0,
\end{equation}
with equality if and only if  $8n+1= M^2$ and all prime divisors of $M$ are congruent to $1$ modulo $4$.
\end{theorem}
\noi
We will use \eqref{hc2} to prove: 
\begin{theorem}
\label{b20}
Let $n$ be a positive integer. Then
\begin{equation}
[q^{8n+1}]\sum_{\substack{x\equiv1\imod{4},\\
                \mbox{ }y\equiv4\imod{16},\\
                z\equiv0\imod{2}}} q^{x^2+ y^2 + 2z^2}\geq0,
\end{equation}
with equality if and only if $8n+1$= $E^2$ and  all prime divisors of $E$ are congruent to $1$ or $3$ modulo $8$.
\end{theorem}

\noi
We will use \eqref{hc3} to prove: 
\begin{theorem}
\label{b30}
Let $n$ be a positive integer. Then
\begin{equation}
 [q^{24n+1}]\sum_{\substack{ x\equiv 3\imod{12},\\
                y\equiv z\imod{6},\\
                y+z\equiv 2\imod{6}}} q^{x^2+ 4y^2 +12z^2}\geq0,
\end{equation}  
with equality if and only if $24n+1=W^2$ and all prime divisors of $W$ are congruent to $1$ modulo $3$.
\end{theorem}

The rest of this paper is organized as follows. In the next section, we collect necessary $q$-series  identities. In Section $3$ we will 
prove Theorems \ref{b10}, \ref{b20} and \ref{b30}. In Section $4$, we prove some conjectures of Kaplansky regarding two spinor-regular
positive ternary forms of discriminant $16384$.

\section{$q$-Series and Theta Function identities}
\label{theta }

We begin by collecting some standard $q$-notations, definitions, and useful formulas,
\begin{equation}
\label{poch}
(a;q)_\i:= \prod_{n=0}^{\infty}(1-aq^n),
\end{equation}
and
\begin{equation}
\label{E}
E(q):= (q;q)_\i.
\end{equation}
Note that 
\begin{equation}
\label{Eneg}
E(-q):= \df{E^3(q^2)}{E(q^4)E(q)}. 
\end{equation}
Next, we recall the Ramanujan theta function
\begin{equation}
\label{fdef}
f(a,b):= \sum_{n=-\infty}^{\infty}a^{\tf{n(n+1)}{2}}b^{\tf{n(n-1)}{2}}, \hspace{1cm} |ab|<1.
\end{equation}
In Ramanujan's notation, the celebrated Jacobi triple product identity takes the shape \cite[p. 35]{bern},
\begin{equation}
\label{jtp}
f(a,b) = (-a;ab)_{\i}(-b;ab)_{\i}(ab;ab)_{\i}.
\end{equation}
One may use \eqref{jtp} to derive the following important special cases:
\begin{equation}
\label{pent}
f(-q,-q^2)=\sum_{n=-\infty}^\infty(-1)^nq^{\tf{n(3n+1)}{2}}=E(q),
\end{equation}
\begin{equation}
   \label{Jac1}
   \sum_{n>0}\left(\df{-4}{n}\right) n q^{n^2}=qE(q^8)^3,
   \end{equation}
	
	\begin{equation}
   \label{Jac2}
   \sum_{n>0}\left(\df{-2}{n}\right) n q^{n^2}=qE(-q^8)^3,
   \end{equation}

	 \begin{equation}
   \label{phi}
   \phi(q):=f(q,q)= \sum_{n=-\infty}^{\infty}q^{n^2} = \df{E^5(q^2)}{E^2(q^4)E^2(q)},
   \end{equation}
	
	 \begin{equation}
   \label{phineg}
   \phi(-q)= \sum_{n=-\infty}^{\infty}(-1)^n q^{n^2} = \df{E^2(q)}{E(q^2)},
   \end{equation}
	
   \begin{equation}
   \label{psi}
   \psi(q):=f(q,q^3)= \sum_{n=-\infty}^{\infty}q^{2n^2-n} = \df{E^2(q^2)}{E(q)},
   \end{equation}
	
	 \begin{equation}
   \label{psineg}
   \psi(-q) = \sum_{n=-\infty}^{\infty}(-1)^n q^{2n^2-n}= \df{E(q^4)E(q)}{E(q^2)},
   \end{equation}
	
\begin{equation}
\label{f12}
f(q,q^2)= \sum_{n=-\infty}^{\infty} q^{\tf{n(3n+1)}{2}}=  \df{E^2(q^3)E(q^2)}{E(q^6)E(q)},
\end{equation}
	
\begin{equation}
\label{f15}
f(q,q^5)= \sum_{n=-\infty}^{\infty} q^{3n^2+2n}= \df{E(q^{12})E(q^3)E^2(q^2)}{E(q^6)E(q^4)E(q)}.
\end{equation}
Note that \eqref{pent} is the famous Euler pentagonal number theorem.  We observe that \eqref{phineg} and \eqref{psi} can be combined to yield
\begin{equation}
\label{lucky}
\phi(-q^8)^2\psi(q^8) = E(q^8)^3. 
\end{equation}		
Analogously, from \eqref{phi}, \eqref{phineg}, \eqref{psi} and \eqref{psineg}, we deduce that 
\begin{equation}
\label{lucky2}
\phi(-q^2)\psi(q) = \phi(q)\psi(-q).
\end{equation}
The relations \eqref{lucky} and \eqref{lucky2} will come in handy in our proof of Theorems \ref{b10} and \ref{b20}, respectively.

The  well-known quintuple product identity can be formulated as in \cite[p. 18]{bern2},
\begin{equation}
\label{qpi}
\sum_{n=-\infty}^{\infty} q^{3n^2+n}\left(\df{z^{3n}}{q^{3n}}-\df{q^{3n+1}}{z^{3n+1}}\right) = 
(q^2;q^2)_{\i}(qz;q^2)_{\i}\left(\df{q}{z};q^2\right)_{\i}(z^2;q^4)_{\i}\left(\df{q^4}{z^2};q^4\right)_{\i},
\end{equation}
with $z\not=0$.
It implies that

\begin{equation}
\label{scqpi}
\sum_{n>0}\left(\df{-12}{n}\right)nq^{n^2} =q\phi(q^{12})E(q^{12})^2=q\df{E(q^{24})^5}{E(q^{48})^2},
\end{equation}

and
\begin{equation}
\label{scqpi2}
\sum_{n>0} (-1)^{n+1} \left(\df{-3}{n}\right)nq^{n^2} =q\phi(q^{3})E(q^{12})^2=q\df{E(q^{6})^5}{E(q^{3})^2}.
\end{equation}

The function $f(a,b)$ may be dissected in many different ways. We will use the following trivial dissections \cite[p. 40, p. 49]{bern},
\begin{equation}
\label{phevod}
\phi(q)=\phi(q^4)+2q\psi(q^8), 
\end{equation}

\begin{equation}
\label{mod31}
\phi(q)= \phi(q^9) +2qf(q^3,q^{15}),
\end{equation}
	
\begin{equation}
\label{mod33}
f(q,q^5)= f(q^8,q^{16})+ qf(q^4,q^{20}).
\end{equation}
Next, we define 
\begin{equation}
a(q):= \sum_{x,y}q^{x^2+xy+y^2}
\end{equation}
and
\begin{equation}
c(q):= \sum_{x,y}q^{x^2+xy+y^2+x+y}.
\end{equation}
These functions were extensively studied in recent literature \cite{bch}, \cite{bern}, \cite{hgb} and \cite{shen}.
I will record below some useful formulas from \cite{berk0}, \cite{bern} and \cite{hgb}:
\begin{equation}
\label{1.17}
\phi(q)^2=\phi(q^2)^2+4q\psi(q^4)^2,
\end{equation}

\begin{equation}
\label{modeqn}
\phi(q)^4-\phi(q^3)^4=8qf(q,q^5)^3\phi(q^3),
\end{equation}

\begin{equation}
\label{evoddis}
\phi(q)\phi(q^3) = a(q^4)+2q\psi(q^2)\psi(q^6),
\end{equation}

\begin{equation}
\label{aq1}
a(q) = a(q^4) + 6q \psi(q^2)\psi(q^6),
\end{equation}

\begin {equation}
\label{aq3}
\psi(q)\psi(q^3) = \psi(q^4)\phi(q^{6})+q\phi(q^2)\psi(q^{12}),
\end{equation}

\begin{equation}
\label{aq2}
4\phi(q^3)\phi(q)a(q^2)-\phi(q)^4 = 3\phi(q^3)^4,
\end{equation}

\begin{equation}
\label{cq2}
c(q)= 3\df{E^3(q^3)}{E(q)}.
\end{equation}
At this point, it is expedient to introduce a projection operator $P_{t,s}$. 
It is defined by its action on power series as follows.
Let $t,s$ be non-negative integers such that $0\leq s<t$. Then
\begin{equation}
P_{t,s}\sum_{n\geq 0}c(n)q^n:  = \sum_{n\geq 0}c(tn +s)q^{tn +s}.
\end{equation}
Next, we observe that
\begin {equation}
\label{neq1}
 P_{2,1} f(q^{3},q^{15})\phi(q^3)=\sum_{x\not\equiv y\imod{2}} q^{(3x+1)^2+3y^2-1}=\sum_{x,y}q^{(3x+1)^2+ 3(x+2y+1)^2-1} = q^3 c(q^{12}).
\end{equation}
Replacing $q^3$ by $q$ in \eqref{neq1}, we find 
\begin {equation}
\label{neq2}
 P_{2,1} f(q,q^{5})\phi(q) = qc(q^4).
\end{equation}
We note in passing, that \eqref{neq2} follows easily from 
\begin {equation}
\label{aq4}
f(q,q^5)\phi(q)= q c(q^4) +  \psi(q^2)f(q^2,q^4),
\end{equation}
but we will not take the time to prove it. 
A motivated reader is invited to examine \cite{bch} and \cite{hgb}  for many similar relations.

With the aid of \eqref{mod31} it is straightforward to verify that 
\begin{equation}
\label{easy}
P_{3,1} \phi(q)^2\phi(q^3) = 4 q f(q^3,q^{15})\phi(q^3)\phi(q^9). 
\end{equation}
Next, we employ \eqref{mod33}, \eqref{evoddis}, \eqref{aq1}, \eqref{aq3}, \eqref{neq2} along with 
\eqref{easy} to deduce
\begin{equation}
\label{hard}
P_{24,1} \phi(q)^2\phi(q^3) = 4q f(q^{24},q^{48})a(q^{48}) + 8 q^{25}\psi(q^{72})c(q^{48}).
\end{equation}
We would like to transform \eqref{hard} into 
\begin{equation}
\label{harder}
P_{24,1} \phi(q)^2\phi(q^3) = 4q \df{E(q^{24})^5}{E(q^{48})^2} + 16 q^{25}\psi(q^{72})c(q^{48}).
\end{equation}
This equation will play an important role in establishing Theorem \ref{b30}.
\begin{proof}
Comparing \eqref{hard} and \eqref{harder} we see that all that is needed to show is   
\begin{equation}
\label{step1}
a(q^2) =   \df{E(q)^5}{f(q,q^2)E(q^2)^2} + 2q\psi(q^3)\df{c(q^2)}{f(q,q^2)}.  
\end{equation}
Replacing $q$ by $-q$ in \eqref{step1} and multiplying both sides by $-4\phi(q^3)\phi(q)$ 
with the aid of \eqref{Eneg}, \eqref{phi}, \eqref{psineg}, \eqref{f12}, \eqref{f15} and \eqref{cq2}, we get   
\begin{equation}
\label{step2}
3\phi(q)^4-  (4\phi(q^3)\phi(q)a(q^2)- \phi(q)^4)  = 24q f(q,q^5)^3\phi(q^3).
\end{equation}
Making use of \eqref{aq2} we arrive at \eqref{modeqn}. The proof of \eqref{harder} is now complete.
\end{proof}

\section{Proof of Theorems \ref{b10}, \ref{b20} and \ref{b30}}
\label{threethm}

We begin by rewriting Theorem \ref{b10} as
\begin{theorem}
\label{b11}
Let $n$ be a positive integer. Then,
\begin{equation}
\label{chui}
[q^{8n+1}]q^9\psi(q^8)\psi(q^{32})^2 \geq0,
\end{equation}
with equality if and only if $8n+1= M^2$ and all prime divisors of $M$ are congruent to $1\imod 4$.
\end{theorem}

\begin{proof}
Replace $q$ by $-q$ in \eqref{1.17} to get 
\begin{equation}
\label{k0}
\phi(-q)^2 = \phi(q^2)^2 - 4q\psi(q^4)^2.
\end{equation}
Subtracting \eqref{k0} from \eqref{1.17} and replacing $q$ by $q ^8$ we find  
\begin{equation}
\label{k1}
\phi(q^8)^2-\phi(-q^8)^2 = 8 q^8\psi(q^{32})^2.
\end{equation}
Hence
\begin{equation}
\label{k2}
\phi(q^8)^2-8q^8\psi(q^{32})^2 = \phi(-q^8)^2.
\end{equation}
Multiplying both sides of \eqref{k2} by $q\psi(q^8)$ and employing \eqref{lucky} we obtain 
\begin{equation}
\label{k3}
q\phi(q^{8})^2\psi(q^8) -qE(q^8)^3 = 8q^9\psi(q^{32})^2\psi(q^8)
\end{equation}
It is straightforward to verify that 
\begin{equation}
\label{zzz}
P_{4,1}\frac{\phi(q)^3}{6} = q \phi(q^4)^2\psi(q^8).
\end{equation}
With the aid of \eqref{1.17} and \eqref{zzz} we obtain 
\begin{equation}
\label{g}
P_{8,1}\frac{\phi(q)^3}{6} = q \phi(q^8)^2\psi(q^8),
\end{equation}
\begin{equation}
\label{zzz1}
P_{8,5}\frac{\phi(q)^3}{24}=q^5 \psi(q^{16})^2\psi(q^8).
\end{equation}

Combining \eqref{Jac1} and \eqref{k3}, \eqref{g} we find that 
\begin{equation}
\label{t1}
P_{8,1}\left(\frac{\phi(q)^3}{6}-\sum_{n>0}\left(\df{-4}{n}\right) n q^{n^2}\right) = 8 q^9\psi(q^{32})^2\psi(q^{8}).
\end{equation}
Therefore, if $8n+1$ is not a perfect square, then 
 \begin{equation}
[q^{8n+1}]q^9\psi(q^8)\psi(q^{32})^2 = \frac{(1,1,1,0,0,0;8n+1)}{48} >0,
\end{equation}
by \eqref{lg0}.
If $8n+1$ = $M^2$, with $M$ being a positive  integer, then
\begin{equation}
\left(\df{-4}{M}\right)=\left(\df{-1}{M}\right)
\end{equation}
and 
\begin{equation}
[q^{M^2}]8q^9\psi(q^8)\psi(q^{32})^2 = \frac{(1,1,1,0,0,0;M^2)}{6} - \left(\df{-1}{M}\right) M \geq0,
\end{equation}
with equality if and only if all prime divisors of $M$ are congruent to $1\imod 4$, by \eqref{hc1}. 
\end{proof}

Next, we turn to Theorem \ref{b20}. We rewrite it as 
\begin{theorem}
\label{b21}
Let $n$ be a positive integer. Then,
\begin{equation}
[q^{8n+1}]q^{17}\phi(q^8)\psi(q^8)\psi(q^{128})\geq0,
\end{equation}
with equality if and only if $8n+1$= $E^2$ and all prime divisors of $E$ are congruent to $1$ or $3$ modulo $8$.
\end{theorem}

\begin{proof}
Replace $q$ by $-q$ in \eqref{phevod} to get 
\begin{equation}
\label{c0}
\phi(-q)= \phi(q^4)-2q\psi(q^8).
\end{equation}
Subtracting \eqref{phevod} from \eqref{c0} and replacing $q$ by $q ^2$ we find 
\begin{equation}
\label{c1}
\phi(-q^2) = \phi(q^2) - 4q^2\psi(q^{16}).
\end{equation}
Multiply both sides \eqref{c1} by $\phi(q)\psi(q)$ and use \eqref{lucky2} on the left to get 
\begin{equation}
\label{c3}
\phi(q)^2\psi(-q) = \phi(q^2)\phi(q)\psi(q) - 4q^2\psi(q^{16})\phi(q)\psi(q).
\end{equation}
Replacing $q$ by $q^8$ in \eqref{c3} and multiplying both sides by $q$ we obtain
\begin{equation}
\label{c4}
q\phi(q^8)^2\psi(-q^8) = q\phi(q^{16})\psi(q^8)\phi(q^8) - 4q^{17}\phi(q^8)\psi(q^8)\psi(q^{128}).
\end{equation}
Clearly \eqref{lucky} implies  that
\begin{equation}
\label{c5}
q\phi(q^8)^2\psi(-q^8) = qE(-q^8)^3,
\end{equation}
and so we established that 
\begin{equation}
\label{id112}
q\psi(q^8)\phi(q^8)\phi(q^{16})- qE(-q^8)^3 = 4q^{17}\phi(q^8)\psi(q^8)\psi(q^{128})
\end{equation}
holds true.
It is straightforward to verify that 
\begin{equation}
\label{gg}
 P_{8,1}\frac{\phi(q)^2\phi(q^2)}{4} = q\psi(q^8)\phi(q^8)\phi(q^{16}).
\end{equation}
Combining \eqref{Jac2} and \eqref{id112}, \eqref{gg} we find that
\begin{equation}
\label{t2}
P_{8,1}\left(\frac{\phi(q)^2\phi(q^2)}{4}- \sum_{n>0}\left(\df{-2}{n}\right) n q^{n^2}\right) = 4q^{17}\phi(q^8)\psi(q^8)\psi(q^{128}).
\end{equation}
And so, if $8n+1$ is not a perfect square, then   
\begin{equation}
\label{g2}
[q^{8n+1}]q^{17}\phi(q^8)\psi(q^8)\psi(q^{128}) = \frac{(1,1,2,0,0,0; 8n+1)}{16}> 0,
\end{equation} 
by \eqref{dick86}.
If $8n+1$ = $E^2$, with $E$ being positive  integer, then 
\begin{equation}
\label{g3}
[q^{E^2}]4q^{17}\phi(q^8)\psi(q^8)\psi(q^{128}) = \frac{(1,1,2,0,0,0;E^2)}{4}-E\left(\df{-2}{E}\right)\ge0,
\end{equation}
with equality if and only if all prime divisors of $E$ are congruent to $1$ or $3$ modulo $8$, by \eqref{hc2}. 
\end{proof}

It remains to prove Theorem \ref{b30}. We begin by rewriting it as 
\begin{theorem}
\label{b31}
Let $n$ be a positive integer. Then,
\begin{equation}
[q^{24n+1}]q^{25}\psi(q^{72})c(q^{48})\geq0,
\end{equation}  
with equality if and only if $24n+1$ = $W^2$ and all prime divisors of $W$ are congruent to $1$ modulo $3$.
\end{theorem}

\begin{proof}
From \eqref{scqpi} and \eqref{harder} we have 
\begin{equation}
\label{JPEG1}
P_{24,1} \left(\phi(q)^2\phi(q^3) -4\sum_{n>0}\left(\df{-12}{n}\right) n q^{n^2}\right) = 16 q^{25}\psi(q^{72})c(q^{48}).
\end{equation}
If $24n+1$ is not a perfect square, then  
\begin{equation}
\label{JPEG2}
[q^{24n+1}]q^{25}\psi(q^{72})c(q^{48})=\df{(1,1,3,0,0,0; 24n+1)}{16}>0,
\end{equation}
by \eqref{dick89}. If $24n+1$= $W^2$, with $W$ being positive integer, then  
\begin{equation}
\label{JPEG3}
[q^{W^2}]4 q^{25}\psi(q^{72})c(q^{48}) = \df{(1,1,3,0,0,0;W^2)}{4}-W\left(\df{-3}{W}\right)\ge0,
\end{equation}
with equality if and only if all prime divisors of $W$ are congruent to $1$ modulo $3$, by \eqref{hc3}. 
\end{proof}

\section{On spinor-regular ternary forms of discriminant $16384$}
\label{four}
 
According to Dickson \cite{dickson}, a positive definite ternary quadratic form is said to be regular if it represents 
all natural integers not excluded by congruences; that is, it represents all integers represented by its genus. 
We now know that there are exactly $913$ regular ternary quadratic forms \cite{jagykap}. 
In \cite{jp}, Jones and Pall discussed genus companions of certain regular diagonal forms. 
These companions are almost regular, but they fail to represent an infinite set of natural numbers in the same square-class. \\
For example, ternary quadratic form $(3,4,9,0,0,0)$ happens to be almost regular. It is the only genus mate of regular form $(1,3,36,0,0,0)$.
To demonstrate the usefulness of a special case of the quintuple product identity ~\eqref{scqpi2}, we digress a little by taking the time to prove the following
\begin{lemma}
Let $n$ be a non-negative integer. Then
 \begin{equation}
\label{4.0}
(3,4,9,0,0,0;n)\ge(1,3,36,0,0,0;n),
\end{equation}
provided $n$ is not of the form $(6j+r)^2$ with $j\in \mathbb{N}$ and $r=1,2$.
Moreover, 
\begin{equation}
\label{4.00}
(1,3,36,0,0,0;(6j+r)^2)>(3,4,9,0,0,0;(6j+r)^2),
\end{equation}
\begin{equation}
\label{4.001}
(3,4,9,0,0,0;(6j+r)^2)\ge0,
\end{equation}
with equality if and only if all prime divisors of $6j+r$ are congruent to $1$ modulo $3$.
\end{lemma}

\begin{proof}
We begin by setting $s=1,m=12$ in [Thm. 2.1, \cite{bp}] to get 

\begin{equation}
\label{4.1}
\sum_{x,y}q^{72x^2 +12xy+y^2 }-\sum_{x,y}q^{72x^2 +60xy +13y^2} = 2qE(q^{12})^2.
\end{equation}
The unimodular transformation $x \mapsto y$, $y \mapsto -x-6y$ yields 
\begin{equation}
\label{4.2}
\sum_{x,y}q^{72x^2 +12xy+y^2 } = \phi(q)\phi(q^{36}).
\end{equation}
Analogously, with the aid of  $x \mapsto -x+y$, $y \mapsto 2x-3y$ we find that
\begin{equation}
\label{4.3}
\sum_{x,y}q^{72x^2 +60xy+13y^2 } = \phi(q^4)\phi(q^9).
\end{equation}
Hence 
\begin{equation}
\label{4.4}
\vartheta(1,3,36,0,0,0,q) - \vartheta(3,4,9,0,0,0,q)  = 2q\phi(q^3)E(q^{12})^2.
\end{equation}
From \eqref{scqpi2} and \eqref{4.4} it follows that 
\begin{equation}
\label{4.5}
(3,4,9,0,0,0;n)=(1,3,36,0,0,0;n),
\end{equation}
if $n$ is not a perfect square, and that  for $m\in\mathbb{N}$ 
\begin{equation}
\label{4.6}
(1,3,36,0,0,0;m^2)-(3,4,9,0,0,0;m^2) = 2 (-1)^{m+1}\left(\df{-3}{m}\right)m. 
\end{equation}
Hence 

\begin{equation}
\label{4.8}
(3,4,9,0,0,0;9m^2)= (1,3,36,0,0,0;9m^2).
\end{equation}
Finally, \eqref{hc4} implies that for $j\in\mathbb{N}$ 
\begin{equation}
\label{4.9}
(3,4,9,0,0,0;(6j+4)^2)> (1,3,36,0,0,0;(6j+4)^2)>0,
\end{equation}

\begin{equation}
\label{4.10}
(3,4,9,0,0,0;(6j+5)^2)> (1,3,36,0,0,0;(6j+5)^2)>0,
\end{equation}

\begin{equation}
\label{4.11}
(1,3,36,0,0,0;(6j+2)^2)>(3,4,9,0,0,0;(6j+2)^2)>0,
\end{equation}
\begin{equation}
\label{4.12}
(1,3,36,0,0,0;(6j+1)^2)>(3,4,9,0,0,0;(6j+1)^2),
\end{equation}
and 
\begin{equation}
\label{4.12v}
(3,4,9,0,0,0;(6j+1)^2)\ge0,
\end{equation}
with equality if and only if all prime divisors of $(6j+1)$ are congruent to $1$ modulo $3$.
\end{proof}
Hence   $3x^2+ 4y^2 +9z^2 $ represents exclusively all positive integers represented by $x^2+ 3y^2 +36z^2 $, which are not of the form $W^2$,
where  $W$ is generated by $1$ and primes congruent to $1$ modulo $3$. Recalling \cite[Table 5]{dickson},  we have
\begin{theorem}
\label{xoxo}
The form $3x^2+ 4y^2 +9z^2 $ represents exclusively all positive integers not of the form\\
$(3m+2),\\(4m+2),\\9^a(9m+6),\\W^2$,\\
where $a,m,W\in\mathbb{N}$ and $W$ is generated by $1$ and primes congruent to $1$ modulo $3$.
\end{theorem}
\noi

Nowadays, the Jones-Pall forms are called spinor regular. The list containing $29$ spinor regular forms was put together in \cite{jagy}. It contains seven original Jones-Pall forms along with eight forms from \cite{behh} and ten forms from \cite{kap}. Jagy contributed four forms to this list. In addition, he checked all positive ternary form with discriminant up to $1,400,000$ looking for all the spinor regular forms, based on a recipe in \cite{chnernst}.
The list of $29$ forms is conjectured to be complete.

All $10$ of Kaplansky's forms in \cite{kap} can be easily handled by the methods discussed in this paper. 
Here, we focus on just two forms  $(9,16,36,16,4,8)$ and $(9,17,32,-8,8,16)$, both of discriminant $16384$.
We begin by proving the following  
\begin{lemma}
\begin{equation}
\label{identity2}
\vartheta(9,16,36,16,4,8 ,q)=\phi(q^{64})^3 + 2q^{16}\psi(q^{128})\phi(q^{64})^2 +8q^{36}\psi(q^{32})\psi(q^{128})^2 +2q^9\psi(q^8)\psi(q^{32})^2.
\end{equation}
\end{lemma}

\begin{proof}
It is easy to check that 
\begin{equation}
\label{identity1}
9x^2+16y^2+36z^2+16yz+4xz+8xy = (x+ 4y+2z)^2+8x^2+32z^2.
\end{equation}
Hence  
\begin{equation}
\label{identity3}
\vartheta(9,16,36,16,4,8 ,q)=\sum_{x,y,z}q^{(2x+4y+2z)^2 + 8(2x)^2+32z^2}+\sum_{x,y,z} q^{(2x+1+4y+2z)^2 + 8(2x+1)^2+32z^2}.
\end{equation}
Next, 
\begin{equation}
\begin{aligned}
\sum_{x,y,z}q^{(2x+4y+2z)^2 + 8(2x)^2+32z^2} & = \sum_{\substack{x\equiv z\imod{2},\\y}}q^{(2x+4y+2z)^2 + 32x^2+32z^2} \\
& + \sum_{\substack{x\not\equiv z\imod{2},\\y}}q^{(2x+4y+2z)^2 + 32x^2+32z^2}.
\end{aligned}
\end{equation}

Using the substitutions $x \mapsto x+z$, $z \mapsto x-z$, $y \mapsto y-x$ and $x \mapsto x+z+1$, $z \mapsto x-z$, $y \mapsto y-x$ we find that
\begin{equation}
\label{identity4}
\begin{aligned}
\sum_{\substack{x\equiv z\imod{2},\\y}}q^{(2x+4y+2z)^2 + 32x^2+32z^2}&=\sum_{x,y,z}q^{16y^2+64x^2+64z^2} \\ 
& = \sum_{x,y,z}q^{64y^2+64x^2+64z^2}+\sum_{x,y,z}q^{16(2y+1)^2+64x^2+64z^2} \\
& = \phi(q^{64})^3 + 2q^{16}\psi(q^{128})\phi(q^{64})^2,
\end{aligned}
\end{equation}
and 
\begin{equation}
\label{identity5}
\begin{aligned}
\sum_{\substack{x\not\equiv z\imod{2},\\y}}q^{(2x+4y+2z)^2 + 32x^2+32z^2}&=\sum_{x,y,z}q^{(4y+2)^2+16(2x+1)^2+16(2z+1)^2} \\ 
& = 8q^{36}\psi(q^{32})\psi(q^{128})^2,
\end{aligned}
\end{equation}
respectively.\\
Hence
\begin{equation}
\label{identity6}
\sum_{x,y,z}q^{(2x+4y+2z)^2 + 8(2x)^2+32z^2} = \phi(q^{64})^3 + 2q^{16}\psi(q^{128})\phi(q^{64})^2+ 8q^{36}\psi(q^{32})\psi(q^{128})^2. 
\end{equation}
Analogously,
\begin{equation}
\label{identity7}
\begin{aligned}
\sum_{x,y,z}q^{(2x+1+4y+2z)^2+8(2x+1)^2+32z^2} & = \sum_{\substack{x\equiv z\imod{2},\\y}}q^{(2x+1+4y+2z)^2 + 8(2x+1)^2+32z^2} \\
& + \sum_{\substack{x\not\equiv z\imod{2},\\y}}q^{(2x+1+4y+2z)^2+8(2x+1)^2+32z^2} \\
& = \sum_{x,y,z}q^{(4y+1)^2+ 4(4x+1)^2+4(4z+1)^2} \\
& + \sum_{x,y,z}q^{(4y+3)^2+ 4(4x+3)^2+4(4z+3)^2} \\
& = 2q^9\psi(q^8)\psi(q^{32})^2.
\end{aligned}
\end{equation}
Combining \eqref{identity3}, \eqref{identity6} and \eqref{identity7}, we arrive at \eqref{identity2}.
\end{proof}

Next, employing \eqref{identity2} together with  \eqref{zzz}, we have the following
\begin{lemma}
\label{u1l}
Let $n$ be a non-negative integer. We have
\begin{equation}
\label{u1a}
(9,16,36,16,4,8; 64n) = (1,1,1,0,0,0;n)\geq0,
\end{equation}
with equality if and only if  $n$ is of the form $4^a(8m+7)$ with $a,m\in\mathbb{N},$

\begin{equation}
\label{u1b}
(9,16,36,16,4,8; 64n+16) = \df{(1,1,1,0,0,0;4n+1)}{3}>0 ,
\end{equation}

\begin{equation}
\label{u1c}
(9,16,36,16,4,8;32n+4) = 8[q^{32n+4}]q^{36}\psi(q^{32})\psi(q^{128})^2\geq0,
\end{equation}

\begin{equation}
\label{u1d}
(9,16,36,16,4,8;8n+1) = 2[q^{8n+1}]q^9\psi(q^8)\psi(q^{32})^2\geq0,
\end{equation}

\begin{equation}
\label{u1e}
(9,16,36,16,4,8;8n+r) = 0, \mbox{\;\;\;\;if } r\in \{3,5,7\}, 
\end{equation}

\begin{equation}
\label{u1f}
(9,16,36,16,4,8;64n+2\tilde{r}) = 0, \mbox{\;\;\;\;if } \tilde{r}\not\in\{0,2,8,18\}.
\end{equation}

\end{lemma}
\noi
Making use of \eqref{chui} with $q\mapsto q$, $q \mapsto q^4$, we establish  
\begin{theorem}
\label{kapjp1}
The form $9x^2+ 16y^2 +36z^2 + 16yz+ 4xz + 8xy$ represents exclusively all positive integers not of the form\\
$ 4^a(8m+7), \\4^a(8m+3),0\leq a\leq2,\\ 4^a(4m+2),0\leq a\leq2,\\4^a(8m+5),0\leq a\leq1,\\M^2,4M^2$,\\
where $a,m,M\in\mathbb{N}$ and $M$ is generated by $1$ and primes congruent to $1$ modulo $4$.
\end{theorem}

We now turn to the second form  $(9,17,32,-8,8,16)$. Again, we start with the easily verifiable identity
\begin{equation}
9x^2+ 17y^2 +32z^2 -8yz+ 8xz + 6xy = (x+3y-4z)^2 + 4(x-y)^2 + 4(x+y+2z)^2, 
\end{equation}
after a bit of labor, to obtain  
\begin{equation}
\begin{aligned}
\vartheta(9,17,32,-8,8,6,q) & =\phi(q^{32})^2\phi(q^{256}) + 2q^{20}\psi(q^{32})\psi(q^{64})^2 + 8q^{80}\psi(q^{128})\psi(q^{256})^2\\
& + 16q^{144}\psi(q^{128})\psi(q^{512})^2 + 4q^{36}\psi(q^{32})\psi(q^{128})^2 + 2q^9\psi(q^8)\psi(q^{32})^2.
\end{aligned}
\end{equation}
Hence, recalling \eqref{zzz1}, we have the following
\begin{lemma}
\label{u2l}
Let $n,m,a\in\mathbb{N}$. We have
\begin{equation}
\label{u2a}
(9,17,32,-8,8,6;32n) = (1,1,8,0,0,0;n)\geq0,
\end{equation}
 with equality if and only if $n$ is of the form $2~4^a(8m+7), 4m+3$ or $2(8m+3)$ by \cite[Table 5]{dickson}
\begin{equation}
\label{u2b}
(9,17,32,-8,8,6;32n+20) = \df{(1,1,1,0,0,0;8n+5)}{12}>0 ,
\end{equation}

\begin{equation}
\label{u2c}
(9,17,32,-8,8,6;128n+80) = \df{(1,1,1,0,0,0;8n+5)}{3}>0 ,
\end{equation}

\begin{equation}
\label{u2d}
(9,17,32,-8,8,6;128n+16) = 16[q^{128n+16}]q^{144}\psi(q^{128})\psi(q^{512})^2\geq0,
\end{equation}

\begin{equation}
\label{u2e}
(9,17,32,-8,8,6;32n+4) = 4[q^{32n+4}]q^{36}\psi(q^{32})\psi(q^{128})^2\geq0,
\end{equation}

\begin{equation}
\label{u2f}
(9,17,32,-8,8,6;8n+1) = 2[q^{8n+1}]q^9\psi(q^8)\psi(q^{32})^2\geq0,
\end{equation}

\begin{equation}
\label{u2g}
(9,17,32,-8,8,6;8n+r) = 0, \mbox{\;\;\;\;if } r\in\{3,5,7\}, 
\end{equation}

\begin{equation}
\label{u2h}
(9,17,32,-8,8,6;64n+2\tilde{r}) = 0, \mbox{\;\;\;\;if } \tilde{r}\not\in\{0,2,8,16,18\}.
\end{equation}

\end{lemma}
\noi
Employing \eqref{chui} with $q\mapsto q$, $q \mapsto q^4$ and $q \mapsto q^{16}$, we arrive at 
 
\begin{theorem}
\label{kapjp2}
The form  $9x^2+ 17y^2 +32z^2 -8yz+ 8xz + 6xy$  represents exclusively all positive integers not of the form\\
$4^a(8m+7),\\4^a(8m+6),0\leq a\leq2,\\4^a(8m+3),0\leq a\leq3,\\ 4^a(8m+2),0\leq a\leq1,\\8m+5,\\M^2,4M^2,16M^2,$\\
where $a,m,M\in\mathbb{N}$ and $M$ is generated by $1$ and primes congruent to $1$ modulo $4$.
\end{theorem}

\section{Acknowledgements}
\label{ack}
   
I am grateful to William Jagy for a most illuminating discussion and  to Shaun Cooper, Michael Hirschhorn, Elizabeth Loew, Frank Patane for a careful reading of the manuscript. I would like to thank George Andrews, who encouraged me to write up these observations.

\end{document}